\documentclass[12pt]{article}

\usepackage{epsfig}
\usepackage{amssymb,amsmath,amsthm, fullpage}
\usepackage{amsfonts}

\def\N{\mathbb{N}}

\def\E{\mathbb{E}}
\def\P{\mathbb{P}}

\def\eps{\varepsilon}

\newtheorem{thm}{Theorem}
\newtheorem{lemma}[thm]{Lemma}

\newtheorem{prop}[thm]{Proposition}

\numberwithin{equation}{section}

\usepackage{xcolor}

\begin{document}

\title{Random Bernstein-Markov factors\footnote{{\bf 2010 Mathematics Subject Classification:} Primary 41A17\ \ Secondary 30E10, 60F05}}

\author{Igor Pritsker and Koushik Ramachandran}
\maketitle

\begin{abstract}
For a polynomial $P_n$ of degree $n$, Bernstein's inequality states that $\|P_n'\| \le n \|P_n\|$ for all $L^p$ norms on the unit circle, $0<p\le\infty,$
with equality for $P_n(z)= c z^n.$ We study this inequality for random polynomials, and show that the expected (average) and almost sure value of
$\Vert P_n' \Vert/\Vert P_n\Vert$ is often different from the classical deterministic upper bound $n$. In particular, for circles of radii less than one,
the ratio $\Vert P_n' \Vert/\Vert P_n\Vert$ is almost surely bounded as $n$ tends to infinity, and its expected value is uniformly bounded for all degrees
under mild assumptions on the random coefficients. For norms on the unit circle, Borwein and Lockhart mentioned that the asymptotic value of
$\Vert P_n' \Vert/\Vert P_n\Vert$ in probability is $n/\sqrt{3},$ and we strengthen this to almost sure limit for $p=2.$ If the radius $R$ of the circle
is larger than one, then the asymptotic value of $\Vert P_n' \Vert/\Vert P_n\Vert$ in probability is $n/R$, matching the sharp upper bound for
the deterministic case. We also obtain bounds for the case $p=\infty$ on the unit circle.
\end{abstract}

\section{Introduction}

\noindent Let $\mathbb{D}\subset\mathbb{C}$ denote the unit disc centered at the origin. Let $P_n(z) = \sum_{k=0}^{n}a_kz^{k}$ be a polynomial of degree $n.$  The classical Bernstein's inequality for the derivative of a polynomial states that

\begin{equation}\label{Bernstein}
\left\Vert P_n' \right\Vert_{\infty} \leq n \left\Vert P_n\right\Vert_{\infty},
\end{equation}
\noindent where $\left\Vert f\right\Vert_{\infty}$ denotes the $\sup$ norm of $f$ over $\partial\mathbb{D}$. Polynomials $P_n(z) = cz^n$ are extremal for the inequality \eqref{Bernstein}. It is also well known that inequality \eqref{Bernstein} holds when $\left\Vert \right\Vert_{\infty}$ is replaced by $L^{p}$ norms $\left\Vert \right\Vert_{p}$ taken over $\partial\mathbb{D}$ (with respect to $d\theta/(2\pi) $), for any $p>0.$ More information about Bernstein's inequality and its history can be found in \cite{RS}. Bernstein's inequality is a direct analogue of Markov's inequality for the sup norm of the derivative of a polynomial on the interval $[-1,1]$, see \cite{RS}. Since Markov's inequality is apparently the first sharp estimate for the derivative of a polynomial, the ratio of norms
$$M_n(P_n) :=\dfrac{\left\Vert P_n' \right\Vert_{\infty}}{\left\Vert P_n \right\Vert_{\infty}}$$
is often referred to as the Markov factor of $P_n,$ while in our case the Bernstein-Markov factor is a more appropriate term. If $P_n(z)\equiv 0$ then we set $M_n(P_n) := 0$, which holds for any constant polynomial, and this convention is carried through to all other norms used in our paper.

\vspace{0.1in}

In this article, we consider random polynomials of the form
$$P_n(z) = \sum_{k=0}^{n}A_kz^{k},$$
where the $\{A_k\}$ i.i.d. complex random variables. Throughout this article we will assume that our random variables are non trivial, meaning that $\P(A_k=0)<1.$ Our goal is to study the asymptotics of the \emph{average} Bernstein-Markov factors, namely
$$\mathbb{E}[M_n(P_n)] = \mathbb{E}\left[\dfrac{\|P_n'\|}{\|P_n\|}\right]$$
for various norms, as well as probabilistic limits of these factors. Define the $L^p$ norm on the unit circle $\partial\mathbb{D}$ by
\[
\|P_n\|_p = \left( \frac{1}{2\pi} \int_0^{2\pi} |P_n(e^{it})|^p\,dt \right)^{1/p},\quad p>0.
\]

\noindent The Bernstein-Markov factors of random polynomials for the $L^p$ norms on the unit circle were considered by Borwein and Lockhart,  who stated an asymptotic for the case of random Littlewood polynomials. But Theorems $1$ and $2$ in \cite{BL} give the following more general result.

\begin{thm} \label{bl}
Let $0 < p < \infty.$ Consider random polynomials $P_n(z) = \sum_{k=0}^{n}A_kz^k$ where the $A_k$ are i.i.d. real valued random variables with mean $0$, variance $1$ and additionally if $p >2,$ $\mathbb{E}\left(|A_0|^{2p} \right) < \infty.$ Then

$$\dfrac{1}{n}\dfrac{\left\Vert P_n'\right\Vert_p}{\left\Vert P_n\right\Vert_p}\xrightarrow{\textbf{P}} \dfrac{1}{\sqrt{3}} \quad\mbox{as } n\to\infty,$$
where the convergence holds in probability. As a consequence, the average Bernstein-Markov factors satisfy
\begin{align*}
\lim_{n\to\infty}\dfrac{1}{n}\mathbb{E}\left[\dfrac{\left\Vert P_n'\right\Vert_p}{\left\Vert P_n\right\Vert_p}\right]=\dfrac{1}{\sqrt{3}}.
\end{align*}
\end{thm}

\noindent We first show that when $p=2$, the convergence of Bernstein-Markov factors holds in the stronger almost sure sense.

\vspace{0.1in}

\begin{thm}\label{1}
If $\{A_k\}$ are i.i.d. complex valued random variables with  $\E[|A_0|^2]<\infty$, then
$$\lim_{n\to\infty}\dfrac{1}{n}\dfrac{\left\Vert P_n'\right\Vert_2}{\left\Vert P_n\right\Vert_2}=\dfrac{1}{\sqrt{3}}\quad \mbox{a.s.}$$
Hence the average Bernstein-Markov factors satisfy
\begin{align*}
\lim_{n\to\infty}\dfrac{1}{n}\mathbb{E}\left[\dfrac{\left\Vert P_n'\right\Vert_2}{\left\Vert P_n\right\Vert_2}\right]=\dfrac{1}{\sqrt{3}}.
\end{align*}
\end{thm}

\medskip We next consider the average Bernstein-Markov factors for the $L^p$ norm on $\partial\mathbb{D}_r = \{z: |z| = r\},$ for $ r>0.$ This norm is defined by
\[
\|P_n\|_{p,r} = \left( \frac{1}{2\pi} \int_0^{2\pi} |P_n(re^{it})|^p\,dt \right)^{1/p},\quad p>0.
\]

\noindent Using a simple change of variable $z\to rz$, we obtain from the Bernstein inequality on the unit circle that
\begin{equation*}
\left\Vert P_n' \right\Vert_{p,r} \leq \frac{n}{r} \left\Vert P_n\right\Vert_{p,r}, \quad 0<p\le\infty,
\end{equation*}
where equality holds for polynomials of the form $P_n(z)=cz^n.$ However, we show that in the randomized model of this problem, the Bernstein-Markov factors remain bounded as $n\to\infty$ with probability one, for all $r\in(0,1).$ At a heuristic level, this is because for $z$ such that $|z| = r$, we can view $P_n(z)$ as a partial sum of a random power series $F(z).$ Hence one expects the limit of $M_n$ to be the corresponding ratio of norms for $F'$ and $F.$
\begin{thm}\label{2}
Suppose that $0<p\le\infty$ and $r\in (0,1)$. If $\{A_k\}$ are i.i.d. complex random variables with $\E[\log^+|A_0|]<\infty$,
then there exists a positive random variable $X_{p,r}$  such that
$$0<\lim_{n\to\infty}\dfrac{\left\Vert P_n'\right\Vert_{p, r}}{\left\Vert P_n\right\Vert_{p, r}}=X_{p,r} < \infty \quad\mbox{a.s.}$$
If $\E[|A_0|]<\infty,$ and for some $\eps>0$ the distribution of $|A_0|$ is absolutely continuous on $[0,\eps]$, with the density bounded on $[0,\eps]$, then
$$\sup_{n\in\N} \mathbb{E}\left[\dfrac{\left\Vert P_n'\right\Vert_{p, r}}{\left\Vert P_n\right\Vert_{p, r}}\right] < \infty.$$
\end{thm}

The second statement above admits a stronger result  when $p=2$. This we collect as a proposition.

\begin{prop}\label{P1}
Let $r\in(0, 1).$ Let $\{A_k\}$ be i.i.d. complex valued random variables with $\mathbb{E}(|A_0|^2)<\infty.$ Then there exists a positive random variable $X_{2,r}$ such that
$$0<\lim_{n\to\infty}\dfrac{\left\Vert P_n'\right\Vert_{2, r}}{\left\Vert P_n\right\Vert_{2, r}}=X_{2,r} < \infty \quad\mbox{a.s.}$$
Furthermore,
$$\lim_{n\to\infty}\mathbb{E}\left[\dfrac{\left\Vert P_n'\right\Vert_{2, r}}{\left\Vert P_n\right\Vert_{2, r}} \right]=\mathbb{E}[X_{2,r}] < \infty.$$
\end{prop}

\medskip
While the Bernstein-Markov factors are typically bounded on circles of radii less than one, we now show that they are asymptotic to $n/R$ (have maximal growth)
for the $L^2$-norms on circles of radii $R>1$.

\begin{thm}\label{3}
 Assume that  $\{A_k\}$ are i.i.d. complex valued random variables with $\mathbb{E}[|A_0|^2] < \infty.$ If for some $\epsilon > 0,$ the distribution of $|A_0|^2$ is absolutely continuous on $[0,\eps]$, and its density is bounded on $[0,\eps]$, then for any $R>1$ we have
$$\dfrac{1}{n}\dfrac{\left\Vert P_n'\right\Vert_{2, R}}{\left\Vert P_n\right\Vert_{2, R}}\xrightarrow{\textbf{P}} \dfrac{1}{R} \quad\mbox{as } n\to\infty.$$
As a consequence, we obtain that
$$\lim_{n\to\infty}\dfrac{1}{n}\mathbb{E}\left[\dfrac{\left\Vert P_n'\right\Vert_{2, R}}{\left\Vert P_n\right\Vert_{2, R}}\right]=\dfrac{1}{R}, \quad R>1.$$
\end{thm}

\bigskip
We return to the supremum norm on the unit circle in the last two results. Results on the asymptotic behaviour of the norm of random Littlewood polynomials are proved in \cite{Ha}, see also \cite{Ka}.

\begin{thm}\label{4}
If $\{A_k\}$ are i.i.d. complex valued random variables with common distribution invariant under complex conjugation,
then the expected Bernstein-Markov factors in the sup norm on the unit circle satisfy
$$\E\left[\dfrac{\left\Vert P_n'\right\Vert_\infty}{\left\Vert P_n\right\Vert_\infty}\right] \ge \dfrac{n}{2}.$$
\end{thm}
\noindent It is clear that the result applies to all real i.i.d. random coefficients, as well as many standard complex ones.

\medskip
Before we state our next theorem, we recall that a random variable $X$ is said to have the standard complex normal distribution, denoted by  $N_{\mathbb {C}}(0,1),$ if it has density $\dfrac{1}{\pi}\exp(-|z|^2)$ in the complex plane $\mathbb{C}$. We note for future reference here that if $X\sim N_{\mathbb {C}}(0,1),$ then $|X|^2\sim E(1),$ the exponential distribution with parameter $1$.

\begin{thm}\label{5}
If $\{A_k\}$ are i.i.d. complex valued random variables with $N_{\mathbb {C}}(0,1)$ distribution, then the expected Bernstein-Markov factors in the sup norm on the unit circle satisfy
$$\dfrac{1}{2}\le\liminf_{n\to\infty}\dfrac{1}{n}\E\left[\dfrac{\left\Vert P_n'\right\Vert_\infty}{\left\Vert P_n\right\Vert_\infty}\right] \le\limsup_{n\to\infty}\dfrac{1}{n}\E\left[\dfrac{\left\Vert P_n'\right\Vert_\infty}{\left\Vert P_n\right\Vert_\infty}\right]\le\sqrt{\dfrac{2}{3}}.$$
\end{thm}

\vspace{0.1in}

\noindent \textbf{Remarks:} The $L^2$ norm of polynomials on the unit circle has an explicit formula in terms of the coeffficients of the polynomial and this helps to derive almost sure convergence of the Bernstein-Markov factors as in Theorem \ref{1} . For general $p\in(0,\infty),\ p\neq 2,$ one does not have such a formula. However by proving convergence in distribution of certain related random variables, Borwein and Lockhart \cite{BL} derive their result. We believe that almost sure convergence holds \emph{for all} $p>0$ and it would be interesting to see a proof of such a result. In  Theorem \ref{3}, we once again expect almost sure convergence although we have not been able to prove it so far. Finally for Theorem \ref{5}, we believe that the limiting value of the scaled average Bernstein-Markov factors exist, and that this limit is $1/\sqrt{3}.$

\section{Proofs}

\begin{proof}[Proof of Theorem \ref{bl}]
Under the conditions stated, Theorems $1$ and $2$ in \cite{BL} give
$$\dfrac{\left\Vert P_n\right\Vert_p}{\sqrt{n}}\xrightarrow{\textbf{P}}\Gamma(1 + p/2)^{\frac{1}{p}} $$
and
$$\dfrac{\left\Vert P_n'\right\Vert_p}{\sqrt{n^3}}\xrightarrow{\textbf{P}}\dfrac{1}{\sqrt{3}}\Gamma(1 + p/2)^{\frac{1}{p}} $$

\noindent where in both cases convegence holds in probability. Combining the two displays gives the first result. Taking expectation and applying the dominated convergence theorem yields the result about convergence of the expected Bernstein-Markov factors.
\end{proof}

\begin{proof}[Proof of Theorem \ref{1}]
It is an easy computation to see that if $P_n(z) = \sum_{k=0}^{n}A_kz^{k}$ then $\left\Vert P_n\right\Vert_2 = \left(\sum_{k=0}^{n}|A_k|^2\right)^{1/2}$ and $\left\Vert P_n'\right\Vert_2 = \left(\sum_{k=1}^{n}k^2|A_k|^2\right)^{1/2}.$ Therefore we obtain that

\begin{equation}\label{L2}
\dfrac{\left\Vert P_n'\right\Vert_2}{\left\Vert P_n\right\Vert_2} = \left(\dfrac{\sum_{k=1}^{n}k^2|A_k|^2}{\sum_{k=0}^{n}|A_k|^2}\right)^{1/2}.
\end{equation}

\noindent By scaling, we can assume without loss of generality that $\mathbb{E}[|A_k|^2] = 1$. With this assumption, since $|A_k|^2$ are i.i.d., the Law of Large Numbers gives
\begin{equation}\label{lln}
\frac{1}{n+1}\sum_{k=0}^{n}|A_k|^2\to 1 \hspace{0.05in}\mbox{a.s.}
\end{equation}

\noindent To deal with the numerator in \eqref{L2} we use the following lemma which is a strengthening of the Law of Large Numbers, cf. Section $5.2$, Theorem $3$ in \cite{CT}.

\begin{lemma}\label{LLN}
Let $\{X_n\}$ be i.i.d. random variables with finite mean. If $\{a_n\}$ are positive constants with $S_n = \sum_{k=1}^{n}a_k \to\infty,$ such that $b_n = S_n/a_n\uparrow\infty$ and
$$b_n^2\sum_{k=n}^{\infty}\frac{1}{b_{k}^2} = \mathcal{O}(n),$$
then
$$\dfrac{1}{S_n}\sum_{k=1}^{n}a_kX_k\to\mathbb{E}[X_1]\ \mbox{ a.s.}$$
\end{lemma}

\noindent Continuing with the proof of Theorem \ref{1}, we apply Lemma \ref{LLN} with $X_k = |A_k|^2,$ $a_k = k^2,$ and $S_n = \sum_{k=1}^{n}k^2$ to obtain

\begin{equation}\label{wlln}
\dfrac{1}{S_n}\sum_{k=1}^{n}k^2|A_k|^2\to 1\hspace{0.05in}\mbox{a.s.}
\end{equation}

Next, we write

\begin{eqnarray}\label{*}
\left(\dfrac{\sum_{k=1}^{n}k^2|A_k|^2}{\sum_{k=0}^{n}|A_k|^2}\right)^{1/2}&= &\sqrt{\dfrac{S_n}{n+1}}\left(\dfrac{\dfrac{\sum_{k=1}^{n}k^2|A_k|^2}{S_n}}{\dfrac{\sum_{k=0}^{n}|A_k|^2}{n+1}}\right)^{1/2}.
\end{eqnarray}

\noindent Dividing equation \eqref{*} by $n$ and then letting $n\to\infty$, we obtain from \eqref{lln} and \eqref{wlln} that

$$\lim_{n\to\infty}\dfrac{1}{n}\dfrac{\left\Vert P_n'\right\Vert_2}{\left\Vert P_n\right\Vert_2}=\dfrac{1}{\sqrt{3}}\hspace{0.05in}\mbox{a.s.}$$

A simple use of the Bounded Convergence Theorem yields
$$\lim_{n\to\infty}\dfrac{1}{n}\mathbb{E}\left[\dfrac{\left\Vert P_n'\right\Vert_2}{\left\Vert P_n\right\Vert_2}\right]=\dfrac{1}{\sqrt{3}}$$ and finishes the proof of Theorem \ref{1}.

\end{proof}

\begin{proof}[Proof of Theorem \ref{2}]
Our assumption $\E[\log^+|A_0|]<\infty$ implies that
\[
\limsup_{n\to\infty} |A_n|^{1/n}\le 1 \quad\mbox{ a.s.,}
\]
see, e.g., \cite{Pr}. Therefore, with probability one, both series
\[
\sum_{k=0}^{\infty}A_k z^k \quad\mbox{and}\quad \sum_{k=1}^{\infty}kA_k z^{k-1}
\]
converge uniformly on compact subsets of the unit disk to the random analytic functions $F\not\equiv 0$ and $F'\not\equiv 0$ respectively.
This means $P_n$ converge uniformly on $|z|=r$ to $F$, and $P_n'$ converge uniformly on $|z|=r$ to $F'$. As a consequence, we obtain that
\[
\lim_{n\to\infty} \|P_n\|_{p,r} =  \|F\|_{p,r} \neq 0 \quad\mbox{and}\quad \lim_{n\to\infty} \|P_n'\|_{p,r} =  \|F'\|_{p,r} \neq 0 \quad\mbox{ a.s.}
\]
Hence the first part of this theorem follows.

For the second part, we need to show that the expected Bernstein-Markov factors are uniformly bounded for all $n\in\N.$ We start with
deterministic estimates for the norms of $P_n$ and $P_n'.$ It is immediate that
\[
\|P_n'\|_{p,r} \le \|P_n'\|_{\infty,r} \le \sum_{k=1}^n k|A_k| r^{k-1},\quad 0<p\le\infty.
\]
If $p\ge 1$ then we use the elementary estimate $\|P_n\|_1 \ge |A_k|$ to give the lower bound
\begin{align*}
\|P_n\|_{p,r} \ge \|P_n\|_{1,r} = \left\| \sum_{k=0}^n A_k r^k z^k \right\|_1 \ge |A_k| r^k,\quad k=0,1,\ldots,n.
\end{align*}
Hence we have for all $P_n$ that
\begin{align*}
\|P_n\|_{p,r} \ge \frac{r}{2} (|A_0| + |A_1|),\quad p\ge 1.
\end{align*}
If $p\in(0,1)$ then we use a theorem of Hardy and Littlewood (cf. Theorem 6.2 in \cite[p. 95]{Du}) to estimate
\begin{align*}
\|P_n\|_{p,r} = \left\| \sum_{k=0}^n A_k r^k z^k \right\|_p \ge c_p r (|A_0|^p + |A_1|^p)^{1/p} \ge c_p r (|A_0| + |A_1|),\quad 0<p<1,
\end{align*}
where $c_p>0$ depends only on $p$. It follows that for all $p\in(0,\infty]$
\begin{align*}
\frac{\|P_n'\|_{p,r}}{\|P_n\|_{p,r}} &\le \frac{\sum_{k=1}^n k|A_k| r^{k-1}}{r \min(c_p,1/2) (|A_0| + |A_1|)}
= \frac{1}{r \min(c_p,1/2)} \left( \frac{|A_1|}{|A_0| + |A_1|} + \frac{\sum_{k=2}^n k|A_k| r^{k-1}}{|A_0| + |A_1|} \right) \\
&\le \frac{1}{r \min(c_p,1/2)} \left( 1 + \frac{\sum_{k=2}^n k|A_k| r^{k-1}}{|A_0| + |A_1|} \right).
\end{align*}
Taking expectation, and using independence of the $A_k$, we obtain that
\begin{align*}
\E\left[\frac{\|P_n'\|_{p,r}}{\|P_n\|_{p,r}}\right] &\le \frac{1}{r \min(c_p,1/2)} \left( 1 + \E\left[\frac{1}{|A_0| + |A_1|}\right]  \E\left[\sum_{k=2}^n k|A_k|r^{k-1}\right] \right) \\
&= \frac{1}{r \min(c_p,1/2)} \left( 1 + \E[|A_0|]\, \E\left[\frac{1}{|A_0| + |A_1|}\right]  \sum_{k=2}^n k r^{k-1} \right).
\end{align*}
Since $\sum_{k=2}^{n} k r^{k-1} \le \sum_{k=2}^{\infty}kr^{k-1}<\infty$ for $0<r<1,$ it remains to show that $\E\left[(|A_0| + |A_1|)^{-1}\right]$ is finite
under the additional assumption on the density $f$ of $|A_0|.$ Let $Z=|A_0|+|A_1|$, and denote the density function of this random variable on $[0,\eps]$ by $g$. It is known that $g$ is given by the convolution $f\ast f$. We finish the proof with the following estimate:
\begin{align*}
\E\left[1/Z\right]  &\le  \int_{0}^{\eps} \left(\frac{1}{x} \int_{0}^{x} f(x-t) f(t)\,dt\right)\,dx + \frac{\P(Z>\eps)}{\eps} \\
&\le \eps \sup_{x\in[0,\eps]} \frac{1}{x} \int_{0}^{x} f(x-t) f(t)\,dt + 1/\eps \\ &\le \eps \sup_{x\in[0,\eps]} f^2(x) + 1/\eps < \infty.
\end{align*}

\end{proof}

\begin{proof}[Proof of Proposition \ref{P1}]
A direct computation shows that $\|P_n\|_{2,r} = \left(\sum_{k=0}^{n}|A_k|^2r^{2k}\right)^{1/2},$ and similarly $\|P_n'\|_{2,r} = \left( \sum_{k=1}^{n}k^2|A_k|^2r^{2k-2}\right)^{1/2}.$ Recall that the Bernstein-Markov factor is $M_n = \Vert P_n'\Vert_{2, r}/\Vert P_n\Vert_{2, r}.$ It is straightforward to see by considering the differences $M_{n+1}- M_{n},$ that $M_n$ increases with $n$ and that  the limit $\lim_{n\to\infty}M_n = X_{2,r}$ exists with probability one, where
$$X_{2,r} = \sqrt{\frac{\sum_{k=1}^{\infty}k^2|A_k|^2r^{2k-2}}{\sum_{k=0}^{\infty}|A_k|^2r^{2k}}}.$$
After applying the monotone convergence theorem, it remains to show that $\mathbb{E}[X_{2,r}] < \infty.$ For this, we first  observe that since $A_0$ is a non-trivial random variable, there exists $\alpha>0$ and $0 < c < 1$ such that
$$\mathbb{P}\left(|A_0|^2 < \alpha\right) = c.$$
Applying the Cauchy-Schwarz inequality, we obtain that $$\mathbb{E}(X_{2,r})\leq\left(\mathbb{E}\left(\sum_{k=1}^{\infty}k^2|A_k|^2r^{2k-2}\right)\right)^{1/2}\left(\mathbb{E}\left(\dfrac{1}{\sum_{k=0}^{\infty}|A_k|^2r^{2k}}\right)\right)^{1/2}.$$ Since $\mathbb{E}(|A_k|^2)<\infty,$ the first term has finite expectation. Hence it is enough to show that $$\mathbb{E}\left(\dfrac{1}{Z_{2,r}}\right) < \infty,$$ where $Z_{2,r} = \sum_{k=0}^{\infty}|A_k|^2r^{2k}.$ For simplicity of notation, we henceforth just write $Z$ for this random variable. As a preparation we claim that
\begin{equation}\label{dist}
\mathbb{P}\left(\alpha r^{2n + 2}\leq Z < \alpha r^{2n} \right)\leq c^{n+1},\quad n\geq 0.
\end{equation}

\noindent Assume for the moment that the estimate \eqref{dist} is true. Then if $ \sqrt{c} < r < 1$, we have

\begin{align*}
\mathbb{E}\left(\dfrac{1}{Z}\right) &= \mathbb{E}\left(\dfrac{1}{Z}; Z > \alpha\right) + \mathbb{E}\left(\dfrac{1}{Z}; 0 < Z < \alpha\right) \\
&\le \frac{1}{\alpha} + \sum_{n=0}^{\infty}\mathbb{E}\left(\dfrac{1}{Z}; \alpha r^{2n + 2} \leq Z < \alpha r^{2n}\right) \\
&\le \frac{1}{\alpha} + \sum_{n=0}^{\infty}\dfrac{1}{\alpha r^{2n+2}}\mathbb{P}\left( \alpha r^{2n + 2} \leq Z < \alpha r^{2n}\right)\\
&\le \frac{1}{\alpha} + \dfrac{1}{\alpha}\sum_{n=0}^{\infty}\left(\dfrac{c}{r^2}\right)^{n+1} < \infty,
\end{align*}
where the last line is true because of our choice of $r.$ So $\mathbb{E}[X_{2,r}] < \infty$ for $\sqrt{c} < r < 1.$ Observe that for a fixed $n$, the quantity $\dfrac{\left\Vert rP_n'\right\Vert_{2, r}}{\left\Vert P_n\right\Vert_{2, r}}$ is monotonically increasing with respect to $r$. This property can be shown by simply computing the derivative with respect to $r$ and checking that it is nonnegative. It follows that for $s\leq\sqrt{c} < r < 1$

$$sX_{2,s} = \lim_{n\to\infty}\dfrac{\left\Vert sP_n'\right\Vert_{2, s}}{\left\Vert P_n\right\Vert_{2, s}}\leq \lim_{n\to\infty}\dfrac{\left\Vert rP_n'\right\Vert_{2, r}}{\left\Vert P_n\right\Vert_{2, r}}=rX_{2,r} $$
\noindent almost surely.  The latter shows $\mathbb{E}[X_{2,r}] < \infty$ for all $r\in(0, 1).$ It now remains to prove \eqref{dist}, which is easily done as follows:

\begin{align*}
\mathbb{P}\left(\alpha r^{2n + 2}\leq Z < \alpha r^{2n} \right) &\leq \mathbb{P}(Z < \alpha r^{2n})\leq \mathbb{P}\left(\sum_{k=0}^{n}|A_k|^2r^{2k} < \alpha r^{2n} \right) \\
&\leq \prod_{k=0}^{n}\mathbb{P}\left(|A_k|^2 <\alpha r^{2n - 2k}\right) \leq \prod_{k=0}^{n}\mathbb{P}\left(|A_k|^2 <\alpha \right) \\
&\leq c^{n+1}.
\end{align*}

\end{proof}

\begin{proof}[Proof of Theorem \ref{3}]
As before, we have $$\dfrac{R\left\Vert P_n'\right\Vert_{2, R}}{\left\Vert P_n\right\Vert_{2, R}} = \sqrt{\dfrac{\sum_{k=1}^{n}k^2|A_k|^2R^{2k}}{\sum_{k=0}^{n}|A_k|^2R^{2k}}}.$$

\noindent Denote $a_k = |A_k|^2$.  After a scaling, we may assume without loss of generality that $\mathbb{E}[a_k] = 1.$ For $j\geq 0,$ let $Y_j = \sum_{k=0}^{j}a_kR^{2k}$. By the Abel summation formula, we obtain that
$$ \sum_{k=1}^{n}k^2a_kR^{2k} = n^2 \sum_{k=0}^{n}a_kR^{2k} - a_0 - \sum_{j=1}^{n-1}(2j +1)Y_j.$$
Therefore
\begin{equation}\label{Abel}
\dfrac{R\left\Vert P_n'\right\Vert_{2, R}}{n\left\Vert P_n\right\Vert_{2, R}}= \sqrt{\dfrac{\sum_{k=1}^{n}k^2a_kR^{2k}}{n^2\sum_{k=0}^{n}a_kR^{2k}}} = \sqrt{1- \dfrac{a_0}{n^2Y_n} - \dfrac{\sum_{j=1}^{n-1}(2j+1)Y_j}{n^2Y_n}}.
\end{equation}

\noindent Since $\dfrac{a_0}{Y_n}\leq 1,$ we see that $\dfrac{a_0}{n^2Y_n}\to 0$ a.s., and hence in probability. We now show that the random variables
$$U_n = \dfrac{\sum_{j=1}^{n-1}(2j+1)Y_j}{n^2Y_n}$$
also converge to $0$ in probability, which proves this theorem. Convergence in probability in turn follows from $\mathbb{E}[U_n^{1/2}]\to 0$ as $n\to\infty.$ Indeed, applying the Cauchy-Schwarz inequality we obtain

\begin{equation}\label{sqrt}
\mathbb{E}[U_n^{1/2}] \leq \left(\sum_{j=1}^{n-1}(2j+1)\mathbb{E}[Y_j]\right)^{1/2} \left(\mathbb{E}\left[\dfrac{1}{n^2Y_n} \right]\right)^{1/2}
\end{equation}

\noindent Since $\mathbb{E}[a_k] = 1,$ we have that $\mathbb{E}[Y_j] = \sum_{k=0}^{j}R^{2k}$. It is a simple matter to now see that $\sum_{j=1}^{n-1}(2j+1)\mathbb{E}[Y_j] = \mathcal{O}\left(nR^{2n}\right).$ To estimate the other term we use the fact that $Y_n\geq R^{2n-2}(a_n + a_{n-1})$ to obtain

\begin{equation}\label{recip2}
 \mathbb{E}\left[\dfrac{1}{n^2Y_n} \right] \leq \dfrac{1}{n^2R^{2n-2}}\mathbb{E}\left[\dfrac{1}{a_n + a_{n-1}} \right] = \mathcal{O}\left(\dfrac{1}{n^2R^{2n}}\right),
\end{equation}
where we used the same idea as in the end of proof for Theorem \ref{2} to conclude that $\mathbb{E}\left[1/(a_n + a_{n-1})\right] < \infty.$ Combining the two estimates and using \eqref{sqrt}, we obtain that $\mathbb{E}[U_n^{1/2}] \to 0$ as $n\to\infty$ and finish the proof.

\end{proof}

\begin{proof}[Proof of Theorem \ref{4}]

For any polynomial $p_n(z)= \sum_{k=0}^{n}a_k z^{k}$ with complex coefficients, we associate its conjugate reciprocal polynomial
$$q_n(z):=z^n \overline{p_n(1/\bar{z})}=\sum_{k=0}^{n}\overline{a}_{n-k} z^{k}.$$
Obviously, $|p_n(z)|=|q_n(z)|$ holds for $|z|=1$. It is also clear that the polynomial $p_nq_n$ is self reciprocal, i.e. coincides with its
conjugate reciprocal. For such polynomials, the Bernstein-Markov factor in the sup norm on the unit circle is known to be equal to the degree
divided by 2, see \cite[p. 527]{RS}. Hence we obtain that
\begin{align*}
n \left\Vert p_n\right\Vert_\infty \left\Vert q_n\right\Vert_\infty &= n \left\Vert p_nq_n\right\Vert_\infty = \left\Vert p_n'q_n+p_nq_n'\right\Vert_\infty \\
&\le  \left\Vert p_n'\right\Vert_\infty \left\Vert q_n\right\Vert_\infty + \left\Vert p_n\right\Vert_\infty \left\Vert q_n'\right\Vert_\infty.
\end{align*}
It follows that
\begin{align} \label{Mrec}
\frac{\left\Vert p_n'\right\Vert_\infty}{\left\Vert p_n\right\Vert_\infty} + \frac{\left\Vert q_n'\right\Vert_\infty}{\left\Vert q_n\right\Vert_\infty} \ge n.
\end{align}
Applying \eqref{Mrec} with $p_n=P_n$ being our random polynomial, and taking into account that the expected values of both additive
terms on the left of \eqref{Mrec} are identical under our assumptions on random coefficients, we arrive at the required result.

\end{proof}

\begin{proof}[Proof of Theorem \ref{5}]

The inequality involving the $\liminf$ of the Bernstein-Markov factors follows immediately from Theorem \ref{4}.  Thus we need to prove the inequality involving the $\limsup$. For this we will give asymptotic bounds for $\left\Vert P_n\right\Vert_\infty$ and $\left\Vert P_n'\right\Vert_\infty$ separately. We start with a bound for $\left\Vert P_n\right\Vert_\infty$.

\begin{prop}\label{p}
 $$\liminf_{n\to\infty}\dfrac{\left\Vert P_n\right\Vert_\infty}{\sqrt{n\log(n)}} \geq 1\hspace{0.05in}\mbox{a.s.}$$
 \end{prop}

 \begin{proof}
 We first observe that for $z$ with $|z|=1,$ $P_n(z)$ is complex Gaussian, distributed as $N_{\mathbb{C}}(0,n+1)$. This means that for such $z,$ $\frac{|P_n(z)|^2}{n+1}\sim E(1)$, where $E(1)$ denotes the exponential random variable with parameter $1.$ For $0\leq l \leq n,$ let $z_l = \exp(\frac{i2\pi  l}{n+1}).$ Note that for $1\le j, l\leq n$ with $j\neq l$ we have

 \begin{equation}\label{indep}
 \mathbb{E}\left(P_n(z_j)\overline{P_n(z_l)}\right) = \sum_{k=0}^{n}(z_j\bar{z_l})^k = 0.
 \end{equation}

 \noindent Hence $P_n(z_j)$ are uncorrelated for different values of $j$. But since they are Gaussian, this means that $P_n(z_j)$, $0\leq j\leq n,$ are independent. This in turn implies the independence of $X_{n,j} = \frac{|P_n(z_j)|^2}{n+1}$, $0\leq j\leq n.$ Therefore

 \begin{equation}\label{exp}
 \dfrac{\left\Vert P_n\right\Vert_\infty ^2}{(n+1)\log(n+1)} \geq\dfrac{ \max_{0\leq j \leq n}X_{n,j}}{\log(n+1)}.
 \end{equation}
\noindent We need the following lemma, whose proof follows by a simple modification of the more well known result involving extremes of i.i.d. exponential random variables, see \cite{Gut}, page 107.

\begin{lemma}\label{max}
Let $Y_{n,j},$ $n\in\mathbb{N},$ $1\leq j \leq n,$ be a triangular array of random variables each of which is distributed as $E(1),$ and in which every row is jointly independent. Let $T_n = \max_{1\leq j\leq n} Y_{n,j}.$ Then

$$\lim_{n\to\infty}\dfrac{T_n}{\log(n)} = 1 \hspace{0.05in}\mbox{a.s.} $$
\end{lemma}

\vspace{0.1in}

\noindent Applying Lemma \ref{max} with $X_{n,j}$ in place of $Y_{n.j}$ and making use of \eqref{exp}, we obtain

$$\liminf_{n\to\infty}\dfrac{\left\Vert P_n\right\Vert_\infty ^2}{(n+1)\log(n+1)}\geq 1 \hspace{0.05in}\mbox{a.s.}$$
\noindent This conculdes the proof of Proposition \ref{p}.
 \end{proof}

Our next objective is to bound $\left\Vert P_n'\right\Vert_\infty$ which we do by means of
\begin{prop}\label{p'}
$$ \limsup_{n\to\infty}\dfrac{\left\Vert P_n'\right\Vert_\infty}{\sqrt{n^3\log(n)}}\leq\sqrt{\dfrac{2}{3}} \hspace{0.05in}\mbox{a.s.}$$
 \end{prop}

 \begin{proof}
 \noindent Note that for each $z$ with $|z|=1,$ $\dfrac{|P_n'(z)|^2}{n(n+1)(2n+1)/6}$ is distributed as $E(1).$ We claim that for every $c > 2$
 $$\sum_{n=1}^{\infty}\mathbb{P}\left(\left\Vert P_n'\right\Vert_\infty^2> c\frac{n(n+1)(2n+1)}{6}\log(n)\right) < \infty.$$
\noindent By the Borel-Cantelli lemma this will prove Proposition \ref{p'}. To show convergence, let us take $\epsilon > 0$ so small that $ c_1= (1-\epsilon)^2c > 2$ and then cover the unit circle by a net of size $\epsilon/(n-1).$ Let $z_{1,\epsilon}, z_{2,\epsilon}....z_{k,\epsilon} $ be these points, where $k = \lceil{\frac{(n-1)}{\epsilon}}\rceil$. We will need the following lemma whose proof is well known, but for the reader's benefit we present it below.

 \begin{lemma}\label{10}
 Let $Q$ be a polynomial of degree $n.$ Let $t\in\partial\mathbb{D}$ be such that $|Q(t)| = \Vert Q\Vert_\infty.$ Let $0 < \epsilon < 1.$ Then for all $z$ such that $|z-t| < \epsilon/n$, we have $|Q(z)|\geq (1-\epsilon)\Vert Q\Vert_\infty.$
 \end{lemma}

 \begin{proof}[Proof of Lemma \ref{10}]
 From the integral representation, we have  $$|Q(z) - Q(t)|\leq |z-t|\left\Vert Q'\right\Vert_\infty\leq |z-t| n \Vert Q\Vert_\infty,$$
 where we used Bernstein's inequality \eqref{Bernstein} in the latter estimate.  Using $|z-t|\leq\epsilon/n,$ and combining with the previous line, the result now follows from
 \begin{align*}
 |Q(z)| & \ge |Q(t)| - |Q(z)- Q(t)| \\
 &\ge \Vert Q\Vert_\infty - \epsilon \Vert Q\Vert_\infty = (1-\epsilon)\Vert Q\Vert_\infty.
 \end{align*}
 \end{proof}

\noindent Using the lemma along with the fact that for $z$ with $|z|=1,$ $\dfrac{|P_n'(z)|^2}{n(n+1)(2n+1)/6}$ is distributed as $E(1),$ we obtain
\begin{align*}
 \mathbb{P}\left(\left\Vert P_n'\right\Vert_\infty^2> c\frac{n(n+1)(2n+1)}{6}\log(n)\right) & \le\mathbb{P}\left(\bigcup_{j=1}^{k}\dfrac{|P_n'(z_{j,\epsilon})|^2}{n(n+1)(2n+1)/6}> c_1\log(n)\right).\\
 &\le \dfrac{k}{n^{ c_1}} \leq C_{\epsilon} \dfrac{1}{n^{ c_1 - 1}}
 \end{align*}
 \noindent and since $ c_1 > 2$, the corresponding series converges and this proves Proposition \ref{p'}.
 \end{proof}

 \noindent Combining Propositions \ref{p} and \ref{p'} yields

 $$\limsup_{n\to\infty}\dfrac{1}{n}\dfrac{\left\Vert P_n'\right\Vert_\infty}{\left\Vert P_n\right\Vert_\infty}\le\sqrt{\dfrac{2}{3}} \hspace{0.1in}\mbox{a.s.}$$

\noindent  An application of Fatou's Lemma gives
$$ \limsup_{n\to\infty}\dfrac{1}{n}\mathbb{E}\left(\dfrac{\left\Vert P_n'\right\Vert_\infty}{\left\Vert P_n\right\Vert_\infty}\right)\le\sqrt{\dfrac{2}{3}},$$

\noindent which is the upper bound in Theorem \ref{5}.

\end{proof}

\textbf{Acknowledgment:} We thank Manjunath Krishnapur for valuable discussions and references.

I. P.: Department of Mathematics,

Oklahoma State University

Stillwater, OK 74078, USA

Email : igor@math.okstate.edu\\

K. R.: Tata Institute of Fundamental Research,

Centre for Applicable Mathematics

Bengaluru, India 560065

Email: koushik@tifrbng.res.in

\end{document}